\documentclass[12pt, english, a4paper]{article}
\usepackage[latin9]{inputenc}
\usepackage[T1]{fontenc}
\usepackage{graphicx}
\usepackage{amsthm}
\usepackage{amsfonts}
\usepackage{amsmath}
\usepackage{amssymb}
\usepackage{makecell}
\usepackage{comment}
\usepackage{cases}
\usepackage{mathrsfs}
\usepackage{fullpage}
\usepackage{times}
\usepackage[usenames]{color}
\usepackage[dvipsnames]{xcolor}
\usepackage{hyperref}

\newcommand{\SSS}{\mathfrak{S}}
\newcommand{\BBB}{\mathfrak{B}}

\newcommand{\AAA}{\mathcal{A}}

\newcommand{\eul}{\mathscr{A}}
\newcommand{\eulb}{\mathscr{B}}

\DeclareMathOperator{\DES}{DES}
\DeclareMathOperator{\EXC}{EXC}
\DeclareMathOperator{\des}{des}

\DeclareMathOperator{\exc}{exc}

\theoremstyle{plain}
\newtheorem{theorem}{Theorem}
\newtheorem{corollary}[theorem]{Corollary}
\newtheorem{lemma}[theorem]{Lemma}

\theoremstyle{definition}
\newtheorem{definition}[theorem]{Definition}

\newtheorem{conjecture}[theorem]{Conjecture}
\theoremstyle{remark}
\newtheorem{remark}[theorem]{Remark}

\usepackage{authblk}
\title{Synchronicity of descent and excedance enumerators in the alternating subgroup}
\author[1]{Umesh Shankar\thanks{\tt{204093001@iitb.ac.in, umeshshankar@outlook.com}}} 
\affil[1]{Department of Mathematics, Indian Institute of Technology, Bombay Mumbai 400076, India} 
\date{\today}
\begin{document}
\maketitle
\begin{abstract}
Generalising the work of Dey, we define the notion of ultra-synchronicity of sequences of real numbers. Let $B_{n,k},C_{n,k},P_{n,k},Q_{n,k}$ be the number of even permutations with $k$ descents, odd permutations with $k$ descents, even permutations with $k$ excedances and odd permutations with $k$ excedances respectively. We show that the four sequences are ultra-synchronised for all $n\ge 5$. This proves a strengthening of two conjectures of Dey.
\end{abstract}
\textbf{\small{}Keyword:}{\small{} alternating group, real-rootedness, descent, excedance, eulerian numbers, log-concavity }{\let\thefootnote\relax\footnotetext{The author is supported by the National Board for Higher Mathematics, India.}}{\let\thefootnote\relax}

\section{Introduction}
A sequence $(a_k)_{k=1}^n$ of reals is said to be log-concave if for each $2\le i\le n-1$, we have $a_{i}^2\ge a_{i+1}a_{i-1}$. We refer the reader to the surveys of Stanley \cite{stanley-survey-logcon} and Branden \cite{branden-unimodality_log_concavity} that showcase the ubiquity of these sequences in combinatorics and geometry and the variety of methods developed to prove log-concavity.  
For a positive integer $n$, denote by $[n]$ the set $\{1,2,\dots,n\}$ and let $\SSS_n$ be the set of permutations on $[n]$. For $\pi=\pi_1\pi_2\dots \pi_n$, define its descent set as $\DES(\pi)=\{i\in [n-1]: \pi_i>\pi_{i+1}\}$ and the descent number as $\des(\pi)=|\DES(\pi)|$. Define the set of excedances as $\EXC(\pi)=\{i\in [n]: \pi_i>i \}$ and the number of excedances as $\exc(\pi)=|\EXC(\pi)|$. We denote the number of permutations in $\SSS_n$ with $k$ descents by $\eul(n,k)$ and the number of permutations in $\SSS_n$ with $k$ excedances by $E(n,k)$. It is well known result of MacMahon \cite{macmahon-book} that the descents and excedances are equidistributed over $\SSS_n$, i.e., $\eul(n,k)=E(n,k)$. Define the Eulerian polynomial, $A_n(t)$, by
\begin{equation}
    A_n(t)=\sum_{\pi\in \SSS_n} t^{\des(\pi)}=\sum_{k=0}^{n-1} \eul(n,k)t^k.
\end{equation}
The Eulerian polynomial $A_n(t)$ enjoys a lot of interesting properties. For example, it is real rooted for all $n$ and the coefficients $\eul(n,k)$ are log-concave.
Let $\AAA_n\subset \SSS_n$ be the set of even permutations. Let $B_{n,k}$ and $C_{n,k}$ be the number of even and odd permutations in $\SSS_n$ with $k$ descents respectively. The difference $$D_{n,k}=B_{n,k}-C_{n,k}$$
were called the signed Eulerian numbers in $\cite{foata-desarm_loday}$.
We have the following recurrences for the numbers $D_{n,k}$ due to Tanimoto \cite{tanimoto-signed}.

\begin{theorem} The signed Eulerian numbers satisfy
\[D_{n,k}=
    \begin{cases}
        (n-k)D_{n-1,k-1}+(k+1)D_{n-1,k} & \mbox{if } n \mbox{ odd}\\
        D_{n-1,k}-D_{n-1,k-1} & \mbox{if } n \mbox{ even}\\        
    \end{cases}\]
\end{theorem}

Similarly, one can define an excedance variant of the signed Eulerian numbers. Let $P_{n,k}$ and $Q_{n,k}$
be number of even and odd permutations in $\SSS_n$ with $k$ excedances respectively. We have the following theorem that was proven by Mantaci \cite{mantaci-jcta-93} and later by Sivasubramanian \cite{siva-exc-det} using determinants of suitably defined matrices. 
\begin{theorem}
    For positive integers $n$ and $0\le k\le n-1$, we have 
    \[
    P_{n,k}-Q_{n,k}=(-1)^k\binom{n-1}{k}.
    \]
\end{theorem}

To show that the sequences $(P_{n,k})_{k=0}^{n-1}$ and $(Q_{n,k})_{k=0}^{n-1}$ are log-concave, Dey \cite{dey-exc-logcon-altgrp} defined the notion of strong synchronisation of sequences.
\begin{definition}
    Two sequences $(X_{k})_{k=0}^{n}$ and $(Y_{k})_{k=0}^{n}$ are said to be strongly synchronised if for $1\le i \le n-1$, we have
    \begin{equation}\label{eqn:strong-sync}
    \min(X_{i},Y_{i})^2\ge \max(X_{i+1},Y_{i+1})\max(X_{i-1},Y_{i-1}),
\end{equation}
\end{definition}
First, we extend this notion of strong synchronisation to the notion of ultra synchronisation.
\begin{definition}
    Two sequences $(X_{k})_{k=0}^{n}$ and $(Y_{k})_{k=0}^{n}$ are said to be ultra-synchronised if for $1\le i \le n-1$, we have
    \begin{equation}\label{eqn:ultra-sync}
    \frac{\min(X_{i},Y_{i})^2}{\binom{n}{i}^2}\ge \frac{\max(X_{i+1},Y_{i+1})}{\binom{n}{i+1}}\frac{\max(X_{i-1},Y_{i-1})}{\binom{n}{i-1}},
\end{equation}
\end{definition}
We extend this notion to multiple sequences $(X^{(j)}_{k})_{k=0}^n$ where $1\le j \le m$.
\begin{definition}
    The sequences $(X^{(j)}_{k})_{k=0}^n$, where $1\le j \le m$, are said to be ultra-synchronised at an index $i$, where $1\le i\le n-1$, if 
    \begin{equation}\label{eqn:multiple-ultra-sync}
         \frac{ \min(X^{(1)}_{i},\dots,X^{(m)}_{i})^2}{\binom{n}{i}^2} \ge \frac{\max(X^{(1)}_{i+1},\dots,X^{(m)}_{i+1})}{\binom{n}{i+1}}\frac{\max(X^{(1)}_{i-1},\dots,X^{(m)}_{i-1})}{\binom{n}{i-1}}  
    \end{equation}
\end{definition}
We call the sequences $(X^{(j)}_{k})_{k=0}^n$, where $1\le j \le m$, ultra-synchronised if they are ultra-synchronised at each index $i$, where $1\le i\le n-1$.
The main result of this paper is the following.
\begin{theorem}\label{thm: main}
   For a positive integer $n \ge 5$, the four sequences $$(B_{n,k})_{k=0}^{n-1}, (C_{n,k})_{k=0}^{n-1},(P_{n,k})_{k=0}^{n-1}, (Q_{n,k})_{k=0}^{n-1}$$ are ultra-synchronised.     
\end{theorem}
\begin{remark}
    The sequences $(P_{n,k})_{k=0}^{n-1}, (Q_{n,k})_{k=0}^{n-1}$ are not ultra-log-concave for $n<5$.
\end{remark}
This confirms {\cite[Conjecture 32]{dey-exc-logcon-altgrp}} and {\cite[Conjecture 33]{dey-exc-logcon-altgrp}} and also, improves it. It is proved in Section \ref{section: proof-main}.
We obtain the following corollary from this theorem.
\begin{corollary}\label{cor-ultra-alt-enum}
For a positive integer $n \ge 5$, the four sequences $$(B_{n,k})_{k=0}^{n-1}, (C_{n,k})_{k=0}^{n-1},(P_{n,k})_{k=0}^{n-1}, (Q_{n,k})_{k=0}^{n-1}$$ are all ultra-log-concave. 
\end{corollary}
This lends further support to a conjecture of Fulman, Kim, Lee and Petersen \cite{fulman2021joint}. 
\begin{conjecture}[Fulman, Kim, Lee, Petersen]
    The polynomials defined by $\sum_{k=0}^{n-1} B_{n,k}t^k$ and $\sum_{k=0}^{n-1} C_{n,k}t^k$ are real-rooted for $n\ge 2$. 
\end{conjecture}

Finally, we give some conjectures about the real rootedness of another class of polynomials.

\section{A real-rooted family of polynomials}

We begin by proving the real-rootedness of a family of polynomials.
\begin{theorem}\label{thm: newrealroots}
    The polynomials defined by 
    \begin{equation}
          P_n(t):=\displaystyle \sum_{k=0}^{n-1} \frac{\eul(n,k)}{\binom{n-1}{k}}t^k
    \end{equation}
    are real-rooted for all $n\in \mathbb N$.
\end{theorem}
We require a few small lemmas to prove Theorem \ref{thm: newrealroots}.
\begin{lemma}
    Define \begin{equation}
        T_n:=\frac{(1+t)}{n-1}\bigg((n-1)+(n-3)t\dfrac{d}{dt}-t^2\dfrac{d^2}{dt^2} \bigg)
        \label{operator}
    \end{equation}
then, we have $T_n(P_{n-1}(t))=P_n(t)$.
\end{lemma}
\begin{proof}
    The numbers $\eul(n,k)$ satisfy the recurrence \begin{equation}
        \eul(n,k)=(k+1)\eul(n-1,k)+(n-k)\eul(n-1,k-1).
        \label{eulerian}
    \end{equation} Write $S(n,k)=\frac{\eul(n,k)}{\binom{n-1}{k}}$ and substitute $\binom{n-1}{k}S(n,k)$ for $\eul(n,k)$ in Equation \ref{eulerian}. We get 
    \begin{equation}
        \binom{n-1}{k}S(n,k)=(k+1)\binom{n-2}{k}S(n-1,k)+(n-k)\binom{n-2}{k-1}S(n-1,k-1).
    \end{equation}
    We simplify this to get 
    \begin{equation}
        S(n,k)=\frac{(n-k-1)(k+1)}{n-1}S(n-1,k)+\frac{k(n-k)}{n-1}S(n-1,k-1)
    \end{equation}
    It can be checked that applying $T_n$ on $P_{n-1}(t)$ produces the RHS of the above equation as the coefficient of $t^k$. This finishes the proof.
\end{proof}
The fact that $T_n$ preserves real rootedness can be proved using the following easy lemma.
\begin{lemma}\label{lem: quickrealderivative}
    If $f\in \mathbb R[x]$ is a polynomial of degree $n$ with non-zero coefficients and real roots, then  $F(x):=nf(x)-xf'(x)$ is real-rooted.
\end{lemma}
\begin{proof}
     Consider $G(x)=x^{n}f(\frac{1}{x})$. The roots of this polynomials are just reciprocals of the roots of $f$. Therefore, it has real roots. Now, $G'(x)=nx^{n-1}f(\frac{1}{x})-x^{n-2}f'(\frac{1}{x})$ also has real roots, but is of degree $n-1$. The reciprocal polynomial of $G'$ is $nf(x)-xf'(x)$ which also has real roots and finishes the proof.
\end{proof}
We can prove Theorem \ref{thm: newrealroots}.
\begin{proof}[Proof of Theorem \ref{thm: newrealroots}]
    We have $$P_n(t)=T_n(P_{n-1})=\frac{(1+t)}{n-1}\bigg((n-1)+(n-3)t\dfrac{d}{dt}-t^2\dfrac{d^2}{dt^2} \bigg)P_{n-1}$$
    This can be rewritten as
    \begin{eqnarray*}
        P_n(t)&=&\frac{(1+t)}{n-1}\bigg(\dfrac{d}{dt}(n+1)tP_{n-1}(t)-\dfrac{d^2}{dt^2}t^2P_{n-1}(t) \bigg)\\
        &=& \frac{(1+t)}{n-1}\dfrac{d}{dt}\bigg(t\big((n-1)P_{n-1}(t)-tP'_{n-1}(t)\big) \bigg)
    \end{eqnarray*}
    We can verify that $P_3(t)=(t+1)^2$ has real roots. Suppose $P_{n-1}(t)$ is real rooted, then so is $(n-1)P_{n-1}(t)-tP'_{n-1}(t)$ by Lemma \ref{lem: quickrealderivative} and therefore, $P_n(t)$.
\end{proof}
\section{Proof of Theorem \ref{thm: main}}\label{section: proof-main}
Before we go into the proof of the main theorem, we need this theorem commonly attributed to Newton (see, Stanley \cite{stanley-survey-logcon}).
\begin{theorem}\label{thm: newton-real-ulc}
    Let $$P(x)=\displaystyle \sum_{k=0}^n \binom{n}{k}a_kx^k$$ be a (real) polynomial with only real roots. Then, the sequence $(a_k)_{k=0}^n$ is log-concave.
\end{theorem}
Define $\varepsilon(i)=\big(\frac{i+1}{i}\big)\big(\frac{n-i}{n-i-1}\big)$.
As a corollary to Theorems \ref{thm: newrealroots} and \ref{thm: newton-real-ulc}, we obtain
\begin{corollary}\label{eqn: required-corollary}
    The Eulerian numbers satisfy the following inequality.
    $$\eul(n,i)^2\ge \varepsilon(i)^2 \eul(n,i-1)\eul(n,i+1).$$
\end{corollary}
From this corollary, we obtain the bounds
\begin{corollary} For positive integer $n$ and $1\le i\le n-1$, we have
    \begin{equation}\label{eqn: strong-logcon-eulerian}
    \eul(n,i)^2-\varepsilon(i)\eul(n,i-1)\eul(n,i+1)\ge  \bigg(\frac{\varepsilon(i)-1}{\varepsilon(i)}\bigg)\eul(n,i)^2
\end{equation}
\end{corollary}
Define 
\[ 
d^{(1)}(n,k)=|B_{n,k}-C_{n,k}|,\hspace{0.4cm} d^{(2)}(n,k)=|P_{n,k}-Q_{n,k}|=\binom{n-1}{k}.
\]

\begin{lemma}\label{lemma-almost}
    The four sequences $$(B_{n,k})_{k=0}^{n-1}, (C_{n,k})_{k=0}^{n-1},(P_{n,k})_{k=0}^{n-1}, (Q_{n,k})_{k=0}^{n-1}$$ are ultra-synchronised at the index $i$, where $1\le i\le n-1$, if the following inequality holds for all $j_1,j_2,j_3\in\{1,2\}$.
    \begin{equation}\label{eqn: to show}
   \bigg(\frac{\varepsilon(i)-1}{\varepsilon(i)}\bigg) \ge 3\varepsilon(i)\frac{d^{(j_1)}(n,i)}{\eul(n,i)}+\varepsilon(i)\frac{d^{(j_2)}(n,i+1)}{\eul(n,i+1)}+2\varepsilon(i)\frac{d^{(j_3)}(n,i-1)}{\eul(n,i-1)}
\end{equation}
\end{lemma}
\begin{proof}
    By the log-concavity of $\eul(n,i)$, we have
\begin{eqnarray}\label{Step1}
    \hspace{-0.7cm}&&3\varepsilon(i)\frac{d^{(j_1)}(n,i)}{\eul(n,i)}+\varepsilon(i)\frac{d^{(j_2)}(n,i+1)}{\eul(n,i+1)}+2\varepsilon(i)\frac{d^{(j_3)}(n,i-1)}{\eul(n,i-1)}\nonumber
    \\\ge\hspace{-0.7cm}  &&\varepsilon(i)\bigg( 3\frac{\eul(n,i)d^{(j_1)}(n,i)}{\eul(n,i)^2}+\frac{\eul(n,i-1)d^{(j_2)}(n,i+1)}{\eul(n,i)^2}+2\frac{\eul(n,i+1)d^{(j_3)}(n,i-1)}{\eul(n,i)^2}\bigg)
\end{eqnarray}
If inequality \eqref{eqn: to show} holds, then the following inequality follows from \eqref{Step1}.
\begin{eqnarray}
   \bigg(\frac{\varepsilon(i)-1}{\varepsilon(i)}\bigg)\eul(n,i)^2 &\ge& \varepsilon(i) \bigg(3\eul(n,i)d^{(j_1)}(n,i)+\eul(n,i-1)d^{(j_2)}(n,i+1)\nonumber\\&&+\ 2\eul(n,i+1)d^{(j_3)}(n,i-1)\bigg)\label{Step2}
\end{eqnarray}
We use the weak bounds $d^{(j_1)}(n,i)\le \eul(n,i)$ and obtain $d^{(j_1)}(n,i)^2\le \eul(n,i)d^{(j_1)}(n,i)$ and $d^{(j_2)}(n,i+1)d^{(j_3)}(n,i-1) \le \eul(n,i+1)d^{(j_3)}(n,i-1)$.
\begin{eqnarray}
  && \hspace{-0.70cm}\varepsilon(i) \bigg(3\eul(n,i)d^{(j_1)}(n,i)+\eul(n,i-1)d^{(j_2)}(n,i+1)+\ 2\eul(n,i+1)d^{(j_3)}(n,i-1)\bigg)\nonumber\\
  &\ge& \varepsilon(i)\eul(n,i+1)d^{(j_3)}(n,i-1)+\varepsilon(i)\eul(n,i-1)d^{(j_2)}(n,i+1)+2\eul(n,i)d^{(j_1)}(n,i)\nonumber\\&&+\ \varepsilon(i)d^{(j_2)}(n,i+1)d^{(j_3)}(n,i-1)-d^{(j_1)}(n,i)^2\label{Step3}
\end{eqnarray}
Combining inequalities \eqref{Step2}, \eqref{Step3} and rearranging terms gives
\begin{eqnarray}
    &&\bigg(\frac{\varepsilon(i)-1}{\varepsilon(i)}\bigg)\eul(n,i)^2+d^{(j_1)}(n,i)^2\nonumber\\ &\ge& \varepsilon(i)\eul(n,i+1)d^{(j_3)}(n,i-1)+\varepsilon(i)\eul(n,i-1)d^{(j_2)}
    (n,i+1)\nonumber\\&&+2\eul(n,i)d^{(j_1)}(n,i)+\varepsilon(i)d^{(j_2)}(n,i+1)d^{(j_3)}(n,i-1)\label{Step4}
\end{eqnarray}
Using Corollary \ref{eqn: required-corollary}, we can write
\begin{equation}
    \eul(n,i)^2-\varepsilon(i)\eul(n,i-1)\eul(n,i+1)+d^{(j_1)}(n,i)^2 \ge \bigg(\frac{\varepsilon(i)-1}{\varepsilon(i)}\bigg)\eul(n,i)^2+d^{(j_1)}(n,i)^2 \label{Step-Cor}
\end{equation}
Combining inequalities \eqref{Step4} and \eqref{Step-Cor} and rearranging gives
\begin{eqnarray}\label{StepFinal}
   && (\eul(n,i)-d^{(j_1)}(n,i))^2\nonumber\\&\ge& \varepsilon(i)(\eul(n,i+1)+\!d^{(j_2)}(n,i+1))(\eul(n,i-1)+\!d^{(j_3)}(n,i-1))
\end{eqnarray}
Since this is true for all choices of $j_1,j_2,j_3$, we have proved our result.
\end{proof}

\begin{remark}
    \[
    \frac{d^{(i)}(n,0)}{\eul(n,0)}=\frac{d^{(i)}(n,n-1)}{\eul(n,n-1)}=1.
    \]
\end{remark}
\begin{lemma}\label{lemma-bounds} When $n\ge 19$, $1\le k \le n-2$ and $i=1,2$,
    \[\frac{d^{(i)}(n,k)}{\eul(n,k)}\le \frac{1}{18n}.
    \]
\end{lemma}
\begin{proof}First, we prove for $i=1$ by induction. The base case at $n=19$ is verified by computer. We assume the hypothesis for $n=n'$ and prove it for $n=n'+1$.
    If $n=n'+1$ is even, then we know that
    \begin{eqnarray*}
       d^{(1)}(n,k)=|D_{n,k}|&=&|D_{n-1,k-1}-D_{n-1,k}|\\ 
        &\le& d^{(1)}(n-1,k-1)+d^{(1)}(n-1,k) \\
    &\le& \frac{\eul(n-1,k-1)}{18(n-1)}+\frac{\eul(n-1,k)}{18(n-1)}\\
    &=& \frac{(n-k)\eul(n-1,k-1)}{18(n-1)(n-k)}+\frac{(k+1)\eul(n-1,k)}{18(n-1)(k+1)}\\
    &\le& \frac{\eul(n,k)}{18(n-1)\min(n-k,k+1)}\\
    &\le& \frac{\eul(n,k)}{18n} \hspace{2cm} (\mbox{as }\min(n-k,k+1)\ge 2)
    \end{eqnarray*}
 Now, if $n=n'+1$ is odd, then we have
 \begin{eqnarray*}
     d^{(1)}(n,k)&=&|(n-k)D_{n-1,k-1}+(k+1)D_{n-1,k}|\\
     &=&|(n-k)(D_{n-2,k-2}-D_{n-2,k-1})+(k+1)(D_{n-2,k-1}-D_{n-2,k})|\\
     &=&|(n-k)D_{n-2,k-2}+kD_{n-2,k-1}-(n-k-1)D_{n-2,k-1}-(k+1)D_{n-2,k}|\\
     &\le &|(n-k)D_{n-2,k-2}|+|kD_{n-2,k-1}|+|(n-k-1)D_{n-2,k-1}|+|(k+1)D_{n-2,k}|\\
     &\le & \frac{(n-k)\eul(n-2,k-2)}{18(n-2)}+\frac{k\eul(n-2,k-1)}{18(n-2)}\\
     &&+\frac{(n-k-1)\eul(n-2,k-1)}{18(n-2)}+\frac{(k+1)\eul(n-2,k)}{18(n-2)}\\
     &=&\frac{\eul(n-1,k-1)}{18(n-2)}+\frac{\eul(n-1,k)}{18(n-2)}\\
     &=&\frac{(n-k)\eul(n-1,k-1)}{18(n-2)(n-k)}+\frac{(k+1)\eul(n-1,k)}{18(n-2)(k+1)}\\
     &\le& \frac{\eul(n,k)}{18(n-2)\min(n-k,k+1)}\\
    &\le& \frac{\eul(n,k)}{18n} \hspace{2cm} (\mbox{as }\min(n-k,k+1)\ge 2)
 \end{eqnarray*}
 We prove the stronger claim that 
 \begin{equation}\label{bound-binom-typeb}
     \frac{\binom{n}{k}}{\eul(n,k)}\le \frac{1}{18n}.
 \end{equation}
We prove once again by induction. The base case at $n=15$ is verified by computer.
\begin{eqnarray*}
       \binom{n}{k}&=&\binom{n-1}{k-1}+\binom{n-1}{k}\\ 
    &\le& \frac{\eul(n-1,k-1)}{18(n-1)}+\frac{\eul(n-1,k)}{18(n-1)}\\
    &=& \frac{(n-k)\eul(n-1,k-1)}{18(n-1)(n-k)}+\frac{(k+1)\eul(n-1,k)}{18(n-1)(k+1)}\\
    &\le& \frac{\eul(n,k)}{18(n-1)\min(n-k,k+1)}\\
    &\le& \frac{\eul(n,k)}{18n} \hspace{2cm} (\mbox{as }\min(n-k,k+1)\ge 2)
    \end{eqnarray*}
 \end{proof}
 \begin{lemma}
     For positive integer $n$,
     $$d^{(1)}(2n+1,1)=\eul(n+1,1)-n,$$
 $$d^{(1)}(2n,1)=\eul(n,1)-n.$$
 \end{lemma}
  \begin{proof}
  We can compute the first few signed Eulerian numbers from Foata, D{\'e}sarm{\'e}nien {\cite[Theorem 1]{foata-desarm_loday}}. 
   We want to calculate the coefficient of $t$ in $(1-t)^{n}\eul_
{n+1}(t)$. It is the same as the coefficient of $t$ in $(1-nt)(1+\eul(n+1,1)t)$. The coefficient is $$d^{(1)}(2n+1,1)=\eul(n+1,1)-n.$$
Similarly, $$d^{(1)}(2n,1)=\eul(n,1)-n.$$ 
\end{proof}
\begin{remark}We have the following explicit formulas for the Eulerian numbers.
     \[
     \eul(n,1)=2^n-n-1;\ \eul(n,2)=3^n-2^n(n+1)+\frac{n(n+1)}{2}.
     \]
\end{remark}
\begin{proof}[Proof of Theorem \ref{thm: main}]
We intend to show that the four sequences $$(B_{n,k})_{k=0}^{n-1}, (C_{n,k})_{k=0}^{n-1},(P_{n,k})_{k=0}^{n-1}, (Q_{n,k})_{k=0}^{n-1}$$ are ultra-synchronised for all positive integers $n\ge 5$. This has been checked for $5 \le n \le 19$ by computer.
For $n \ge 19$, for the indices $2\le k\le n-3$, the ultra-synchronisation follows from combining Lemmas \ref{lemma-almost} and \ref{lemma-bounds}. We only have to verify the claim for the index $1$ and $n-2$. However, due to the symmetry properties of the sequences found in \cite{tanimoto-signed} and \cite{siva-dey-gamma_positive_excedances_alt_group}, we only need to prove this for the index $1$.
As a final step, we have to verify that, for $i,j\in \{1,2\}$,
\[
(\eul(n,1)-d^{(i)}(n,1))^2\ge 2\varepsilon(1)(\eul(n,2)+d^{(j)}(n,2))
\]
We show the calculations for even $n$ and omit the calculations for odd $n$ as they are similar.

Suppose $n=2n'$ is even, from the formulas for $d^{(1)}(n,1)$ and $d^{(2)}(n,1)$, it can be seen, for $n\ge 4$, that $$2^{n'}-2n'-1= d^{(1)}(2n',1)\ge d^{(2)}(2n',1)=2n'-1.$$
We just need to show that
\[
(\eul(2n',1)-d^{(1)}(2n',1))^2\ge 2\varepsilon(1)(\eul(2n',2)+d^{(1)}(2n',2))
\]
The LHS can be written as 
\[
(2^{2n'}-2n'-1-(2^{n'}-2n'-1))^2=(2^{2n'}-2^{n'})^2 \ge \frac{2^{4n'}}{4}
\]
For the RHS, we have the string of inequalities.
\begin{eqnarray*}
    2\varepsilon(1)(\eul(2n',2)+d^{(1)}(2n',2))&\le& 12\eul(2n',2) \le 12(9^{n'}+\binom{2n'}{2}) \\&\le& \frac{2^{4n'}}{4} \le (\eul(2n',1)-d^{(1)}(2n',1))^2.
\end{eqnarray*}
We have used the facts that $\varepsilon(1)\le 3$ and $\eul(2n',2)\ge d^{(1)}(2n',2)$. Finally, $12(9^{n'}+\binom{2n'}{2})\le \frac{2^{4n'}}{4}$ for $n'\ge 6$. This completes the proof.
\end{proof}
\section{Open problems}

We extend the set of conjectures proposed by Fulman, Kim, Lee and Petersen.
Let $P_{n,k}$ and $Q_{n,k}$
be number of even and odd permutations in $\SSS_n$ with $k$ excedances respectively.
\begin{conjecture}
    The polynomials $\sum_{i=0}^n P_{n,k}t^k$ and $\sum_{i=0}^n Q_{n,k}t^k$ are all real-rooted for $n\ge 5$.
\end{conjecture}
In the light of Corollary \ref{cor-ultra-alt-enum}, the coefficients are ultra-log-concave which supports the conjecture. If the conjecture were true, then an application of Theorem \ref{thm: newton-real-ulc} to the aforementioned polynomials gives us another proof of Corollary \ref{cor-ultra-alt-enum}.

\bibliographystyle{acm}
\end{document}